\documentclass{amsart}
\usepackage{calc}
\usepackage{amsmath,amssymb,tabularx}

\theoremstyle{plain}

\newtheorem{theorem}{Theorem}[section]

\newtheorem{proposition}[theorem]{Proposition}

\newtheorem{lemma}[theorem]{Lemma}

\newtheorem{corollary}[theorem]{Corollary}

\newtheorem{remark}[theorem]{Remark}

\theoremstyle{definition}

\newtheorem{definition}[theorem]{Definition}

\begin{document}

\title{Ergodic Properties of Tame Dynamical Systems}


\author{A.V. Romanov}


\address{National Research University Higher School of Economics, Moscow, Russia}

 \email{av.romanov@hse.ru}


\makeatletter
\def\blfootnote{\gdef\@thefnmark{}\@footnotetext}
\makeatother

\blfootnote{\textup{2010} \textit{Mathematics Subject Classification}:
Primary 37A30, 47A35; Secondary 20M20}

\maketitle

\begin{abstract}
We study the problem on the weak-star decomposability of a
topological $\mathbb{N}_{0}$-dynamical system
$(\Omega,\varphi)$, where $\varphi$ is an endomorphism of a
metric compact set~$\Omega$, into ergodic components in
terms of the associated enveloping semigroups. In the tame
case (where the Ellis semigroup $E(\Omega,\varphi)$
consists of $B_{1}$-transformations $\Omega\rightarrow
\Omega$), we show that (i) the desired decomposition exists
for an appropriate choice of the generalized sequential
averaging method; (ii) every sequence of weighted ergodic
means for the shift operator $x\rightarrow x\circ\varphi$,
$x\in C(\Omega)$, contains a pointwise convergent
subsequence. We also discuss the relationship between the
statistical properties of~$(\Omega,\varphi)$ and the mutual
structure of minimal sets and ergodic measures.
\end{abstract}

\maketitle

\section{Introduction}\label{s1}

We are interested in topological $\mathbb{N}_{0}$-dynamical
systems, that is, semicascade $(\Omega,\varphi)$ generated by a
continuous endomorphism~$\varphi$ of a metric compact
set~$\Omega$. The aim of the present paper is to develop a common
point of view on the following three aspects of the theory of such
systems:
\begin{itemize}
\item[(1)] The weak-star convergence of various ergodic means
(averages along the orbits of the system) for scalar test
functions $x\in X\doteq C(\Omega)$ or Radon measures $\mu
\in X^{*}$. (For the case of Ces\'aro means, this approach
goes back to Kryloff--Bogoliouboff~\cite{17} and Oxtoby~\cite{19}.)
\item[(2)] Relations between minimal sets and ergodic measures.
\item[(3)] The decomposability of the dynamical system
$(\Omega,\varphi)$ into irreducible (ergodic) subsystems
depending on the choice of the averaging method.
\end{itemize}

The main results are obtained for the class of tame systems
introduced (under a different name) by K\"{o}hler~\cite{15} and
studied in detail in the papers~\cite{6,7,9,10,14}. There are
several equivalent definitions of tame dynamical systems; for
example, one says that a system $(\Omega,\varphi)$ is tame,
$(\Omega,\varphi)\in \mathcal{D}_{\mathrm{tm}}$, if its Ellis
semigroup consists of endomorphisms of~$\Omega$ belonging to the
first Baire class. The interest in such objects is due to the
relatively simple topology of their enveloping semigroups often
combined with a pretty complex phase dynamics. A number of
assertions on the convergence of generalized ergodic means for
$(\Omega,\varphi)\in \mathcal{D}_{\mathrm{tm}}$ were established
in~\cite{16,21}, where the paper~\cite{16} deals with the more
general case in which arbitrary amenable operator semigroups act
on~$X$. There are reasons to believe that the tame--untame
dichotomy is somehow related to the absence or existence of
chaotic phase dynamics. In any case, every untame semicascade on
$[0,1]$ proves to be chaotic in the sense of Li--Yorke~\cite{15}.

We discuss the following properties of weak-star ergodic (see
Sec.~\ref{s2.1}) operator nets (or sequences) $\mathcal{V}\subset
\mathcal{L}(X^{*})$ identified with the corresponding averaging
methods.
\begin{itemize}
\item[(a)] The convergence of \textit{all} nets (
sequences)~$\mathcal{V}$; the convergence of \textit{some} ergodic
sequences~$\mathcal{V}$.
\item[(b)] The possibility of a statistical description of the
behavior of orbits of $(\Omega,\varphi)$ with the use of
ergodic measures.
\end{itemize}
Properties~(a) of a discrete semiflow $(\Omega,\varphi)$ are
considered in connection with the following dynamic
characteristics:
\begin{itemize}
\item[(i)] The orbital subsystems are uniquely ergodic.
\item[(ii)] The supports of ergodic measures are minimal.
\item[(iii)] The minimal subsystems are uniquely ergodic.
\end{itemize}
Section~\ref{s3} classifies and strengthens the corresponding
results contained in the recent papers~\cite{16,20,21}.
Theorem~\ref{t3.3} establishes that if $(\Omega,\varphi)\in
\mathcal{D}_{\mathrm{tm}}$, then every ergodic sequence~
$\mathcal{V}$ contains a convergent subsequence; in particular,
there exists a convergent subsequence of Ces\'aro means.

The main results of the paper are gathered in Sec.~\ref{s4}. We
show (Theorem~\ref{t4.5}) that a tame dynamical system
$(\Omega,\varphi)$ has distinct (depending on the choice of a
sequential averaging method) decompositions into ergodic
components and describe all such decompositions in terms of some
operator semigroup $\mathcal{K}_{c}\subseteq \mathcal{L}(X^{*})$
related to $(\Omega,\varphi)$. In one interpretation, the
decomposability of $(\Omega,\varphi)\in \mathcal{D}_{\mathrm{tm}}$
into ergodic components implies the existence of an ergodic
sequence $\mathcal{V}$ such that the asymptotic
$\mathcal{V}$-distributions of all orbits are determined by
ergodic measures. Thus, tame semicascades possess property~(b).

Section~\ref{s5} contains a short survey of some typical examples
of tame and untame $\mathbb{N}_{0}$-systems. In particular, we
present the recent results~\cite{18} on the efficient tame--untame
dichotomy for affine semicascades on the tori $\mathbb{T}^{d}$,
$d\geq 1$.

\medskip
Most of the results of the paper were presented by the
author at the 5th Miniworkshop on Operator Theoretic
Aspects of Ergodic Theory held at the Eberhardt Karls
Universit\"{a}t T\"{u}bingen, November 17--18, 2017.

\section{Preliminaries}\label{s2}

We deal with semicascades $(\Omega,\varphi)$, where $\varphi$ is a
continuous endomorphism of a metric compact set~$\Omega$.
Given~$\Omega$, we sometimes identify $(\Omega,\varphi)$ with
$\varphi$ when using terms like ``minimal endomorphism.'' Let
$X=C(\Omega)$, let $Ux=x\circ\varphi$, $x\in X$, be the Koopman
operator, and let $V=U^{*}\in \mathcal{L}(X^{*})$. We have
$\|U\|_{\mathcal{L}(X)}=\|V\|_{\mathcal{L}(X^{*})}=1$. By
$\mathcal{P}(\Omega)$ we denote the convex set of Borel
probability measures on~$\Omega$, which is compact in the
$w^{*}$-topology of the space~$X^{*}$, and by~$X_{1}$ we denote
the subspace of~$X^{**}$ formed by bounded functions of the first
Baire class. Let us present necessary information on ergodic
means, enveloping semigroups associated with $(\Omega,\varphi)$,
and tame dynamical systems.

\subsection{Ergodic means}\label{s2.1}

We slightly modify the classical definition~\cite{2} for the case
of a cyclic semigroup $\{V^{n}\}$ of shift operators and say that
a net $\{V_{\alpha}\} \subseteq \operatorname{co}\{V^{n}$, $n\in
\mathbb{N}_{0}\}$ in $\mathcal{L}(X^{*})$ is ergodic if
\begin{equation}\label{e2.1}
(\operatorname{Id}-V)V_{\alpha}
\overset{\mathrm{W^*O}}{\longrightarrow}0:\quad
(x,(\operatorname{Id}-V)V_{\alpha}\mu)\rightarrow 0,\quad
x\in X,\quad \mu\in X^{*}.
\end{equation}
Here $V_{\alpha}=U_{\alpha}^{*}$, $U_{\alpha} \in \mathcal{L}(X)$,
and the net $\{U_{\alpha}\}\subseteq \operatorname{co}\{U^{n}$,
$n\in \mathbb{N}_{0}\}$ is said to be ergodic as well. If
$V_{\alpha}\overset{\mathrm{W^*O}}{\longrightarrow}Q$, $Q\in
\mathcal{L}(X^{*})$, then $Q^{2}=Q$. In view of the duality
$(Ux,\mu)=(x,V\mu)$, $x\in X$, $\mu \in X^{*}$, the convergence
$V_{\alpha}\overset{\mathrm{W^*O}}{\longrightarrow}Q$ in
$\mathcal{L}(X^{*})$ is equivalent to the convergence
$U_{\alpha}x\overset{w^{*}}{\longrightarrow}Q^{*}x$ in $X^{**}$,
where $x\in X$ and $Q^{*}\in \mathcal{L}(X^{**})$. For ergodic
sequences $\{U_{n}\} \subset \mathcal{L}(X)$, this convergence is
equivalent to the pointwise convergence of functions, $U_{n} x\to
\overline{x} \in X_{1}$. Note that ergodicity is inherited when
passing to subnets and subsequences. When speaking of ergodic
means in what follows, we most often mean operator nets (or
sequences) in~$\mathcal{L}(X^{*})$.

One can obtain various ergodic sequences
$\mathcal{V}=\{V_{n}\}\subset \mathcal{L}(X^{*})$ based on
summation methods for numerical sequences with an infinite
numerical matrix $S=\{s_{n,k}\}$ satisfying the following
conditions:
\begin{itemize}
\item[(1)] $s_{n,k}\geq 0$ and
$\sum_{k=0}^{\infty}s_{n,k}=1$ for each $n\geq 0$.
\item[(2)] Each row of $S$ contains finitely many entries
$s_{n,k}>0$.
\item[(3)] $\lim_{n\rightarrow\infty} (s_{n,0}
+\sum_{k=1}^{\infty}|s_{n,k}-s_{n,k-1}|)=0.$
\end{itemize}
The sequence of operators
$V_{n}=\sum_{k=0}^{\infty}s_{n,k}V^{k}$ proves to be
ergodic, because \break
$\|(\operatorname{Id}-V)V_{n}\|_{\mathcal{L}(X^{*})}\rightarrow
0$ as $n\rightarrow\infty$. For example, the weights
corresponding to appropriate Riesz means have the form
\begin{equation*}
 s_{n,k}=\dfrac {p_{k}}{p_{0}+p_{1}+...+p_{n}}\quad(0\leq k
\leq n),\quad s_{n,k}=0\quad(k>n),
\end{equation*}
where $p_{n}\geq p_{n+1}>0$ and
$\sum_{n=0}^{\infty}p_{n}=\infty$. (One has $p_{n}\equiv 1$
for the Ces\'aro means.)

\subsection{Enveloping semigroups}\label{s2.2}

The Ellis semigroup $E(\Omega,\varphi)$ of a semicascade
$(\Omega,\varphi)$ is the closure of the set $\{\varphi^{n}$,
$n\in \mathbb{N}_{0}\}$ of transformations in the topology of the
direct product $\Omega ^{\Omega }$ \cite{4}. The K\"ohler
semigroup $\mathcal{K}(\Omega,\varphi)$ is the closure of the set
$\mathcal{K}^{0}=\{V^{n}$, $n\in \mathbb{N}_{0}\}$ of operators in
the W$^*$O-topology of the space $\mathcal{L}( X^{*})$ \cite{15}.
Finally, the semigroup $\mathcal{K}_{c}(\Omega,\varphi)$ is
defined as the W$^*$O-closure of the convex hull
$\operatorname{co} \mathcal{K}^{0}$ \cite{20}. The
right-topological semigroups $E(\Omega,\varphi)$,
$\mathcal{K}(\Omega,\varphi)$, and
$\mathcal{K}_{c}(\Omega,\varphi)$ are compact. Actually,
$\mathcal{K}_{c}(\Omega ,\varphi )$ is the enveloping semigroup of
the action $\mathcal{P}\times W\, \stackrel{V}{\longrightarrow}
\mathcal{P}$ on $\mathcal{P}=\mathcal{P}(\Omega )$ of the abelian
semigroup of polynomials $W=\mathrm{co}\{\, t^{n},\, n\ge 0\}$
with the usual multiplication.

Let us present some useful properties of the semigroup
$\mathcal{K}_{c}=\mathcal{K}_{c}(\Omega,\varphi)$
(see~\cite[Sec.~1]{20}). The nonempty kernel $\operatorname{Ker}
\mathcal{K}_{c}$ (the intersection of two-sided ideals) of the
semigroup $\mathcal{K}_{c}$ consists precisely of unit norm
projections $Q \in \mathcal{K}_{c}$ such that $VQ=Q$, or,
equivalently, $QX^{*} = \operatorname{fix}(V)\doteq \{\mu\in
X^{*}:V\mu=\mu\}$. A net $V_{\alpha}\subset \operatorname{co}
\mathcal{K}^{0}$ such that
$V_{\alpha}\overset{\mathrm{W^*O}}{\longrightarrow}T\in
\mathcal{K}_{c}$ is ergodic if and only if
$T\in\operatorname{Ker}\mathcal{K}_{c}$. Every element $Q\in
\operatorname{Ker} \mathcal{K}_{c}$ is the limit of some ergodic
net of operators; i.e., there exist $\mathrm{W^*O}$-convergent
ergodic nets
 for any $\varphi\in
C(\Omega,\Omega)$.

\begin{remark}\label{t2.1}
According to \cite[Theorem~3.2]{20}, all ergodic nets~\eqref{e2.1}
converge if and only if $\operatorname{Ker} \mathcal{K}_{c}$
consists of a single element, which is necessarily the zero
element of the semigroup $\mathcal{K}_{c}$. The paper~\cite{16}
uses a slightly different (wider and more traditional) definition
of ergodic nets; namely, it is assumed that $V_{\alpha}\in
\operatorname{\overline{co}} \mathcal{K}^{0}=\mathcal{K}_{c}$
in~\eqref{e2.1}. Nevertheless, the condition $\operatorname{card}
\operatorname{Ker} \mathcal{K}_{c}=1$ also implies the convergence
of all nets of this kind \cite[Theorem~4.3]{16}.
\end{remark}

\subsection{Tame dynamical systems}\label{s2.3}

Tame $\mathbb{N}_{0}$-systems can be defined as follows in
function-theoretic terms (see~\cite{15}).

\begin{definition}\label{t2.2}
One says that a semicascade $(\Omega,\varphi)$ is tame
($(\Omega,\varphi)\in \mathcal{D}_{\mathrm{tm}}$) if, for any
$x\in X$ and any subsequence $\{ n(k)\}\subseteq \mathbb{N}_{0}$,
\begin{equation*}
 \inf_{a}\biggl\|\sum _{k=0}^{\infty}a_{k}x_{n(k)}\biggr\|_{X}=0,
\end{equation*}
where $x_{n(k)}=x\circ\varphi ^{n(k)}$, the sequences $a\in l^{1}$
have finitely many nonzero terms, and
$\sum_{k=0}^{\infty}\left|a_{k} \right|=1$.
\end{definition}

Essentially, this condition is related to the problem on the
isomorphic embeddability of $l^{1}$ in Banach spaces, which goes
back to Rosenthal~\cite{22}. Let $\Pi_{b}$ and $\Pi_{1}$,
respectively, be the sets of Borel endomorphisms and first Baire
class endomorphisms of~$\Omega$. Each of the following properties
is equivalent to Definition~\ref{t2.2}:
\begin{itemize}
\item[(a)] $E(\Omega ,\varphi )$ is a Fr\'echet--Urysohn compact
set.
\item[(b)] $\operatorname{card} E(\Omega,\varphi ) \le \mathfrak {c}$.
\item[(c)] $\mathcal{K}_{c}(\Omega ,\varphi )$ is a
Fr\'echet--Urysohn compact set.
\item[(d)] $E(\Omega,\varphi)\subset \Pi_{1}$.
\item[(e)] $E(\Omega,\varphi)\subset \Pi_{b}$.
\end{itemize}
Properties~(c) and~(e) as equivalent definitions of a tame
dynamical system arose in \cite[Proposition~3.11]{16} and
\cite[Theorem~3.4]{21}, respectively; the remaining properties can
be found in~\cite{6,7}. Essentially, the semigroups
$E(\Omega,\varphi)$ and $\mathcal{K}_{c}(\Omega,\varphi)$ in
conditions~(a) and~(c) are sequentially compact. According to~(a),
a dynamical system is tame if its Ellis semigroup is metrizable.
Compact subsystems and direct products of tame systems prove to be
tame themselves~\cite{6}.

\section{Convergence of Ergodic Means}\label{s3}

\noindent A criterion for the weak-star convergence of
Ces\'aro means
\begin{equation*}
 U_{n}=\frac{1}{n+1} (I+U+...+U^{n}),\quad\quad
V_{n}=\frac{1}{n+1} (I+V+...+V^{n})
\end{equation*}
was obtained in~\cite[Theorem~1]{11} and extended to arbitrary
ergodic nets $\{U_{\alpha}\}\subset \mathcal{L}(X)$ and
$\{V_{\alpha}\}\subset \mathcal{L}(X^{*})$
in~\cite[Theorem~1.5]{20}. Namely, the following theorem holds.

\begin{theorem}[the separation principle]\label{t3.1}
Let $X_{0}=\{x\in X: U_{\alpha}x
\overset{w^{*}}{\longrightarrow} \overline{x} \in
X^{**}\}$. One has $X_{0}=X$ if and only if the limit
elements $\overline{x}$ separate $\operatorname{fix}(V)$.
\end{theorem}

The latter condition means that for each invariant measure
$\mu=V\mu,\;\mu\in X^{*},$ there exist continuous functions
$x_{1},x_{2}\in X_{0}$ such that $(\overline{x}_{1},\mu)\neq
(\overline{x}_{2},\mu)$. Further, $X_{0}$ is a nonempty closed
$U$-invariant linear subspace of~$X$, $\overline{x}=Tx$,
$T\in\mathcal{L}(X_{0},X^{**})$, and $\|T\|=1$. In the case of
ergodic sequences, one has $\overline{x} \in X_{1}$.

We use the following notation for an $\mathbb{N}_{0}$-dynamical
system $(\Omega,\varphi)$: $m\subseteq \Omega$ is a minimal set;
$\mu_{e}\in \mathcal{P}(\Omega)$ is an ergodic measure;
$\overline{o}(\omega)$ is the closure of the orbit
$o(\omega)=\{\varphi^{n}\omega$, $n\geq0\}$ of an element
$\omega\in \Omega$. We are interested in the following dynamic
properties of $(\Omega,\varphi)$ and ergodic operator nets
$\mathcal{V}\subset \mathcal{L}(X^{*})$:
\begin{description}
    \item[(single $m$ in $\overline{o}$)] Every
$\overline{o}(\omega)$ contains a single $m$.
    \item[($\operatorname{supp} \mu_{e}=m$)]
    The supports of $\mu_{e}$ are
minimal.
    \item[(single $\mu_{e}$ on $m$)] The minimal subsystems
$(m,\varphi)$ are uniquely ergodic.
    \item[$\operatorname{UE} (\overline{o})$] The orbital subsystems
$(\overline{o}(\omega),\varphi)$ are uniquely ergodic.
    \item[(AEN)] All sets $\mathcal{V}$ converge.
    \item[(AES)] All sequences $\mathcal{V}$ converge.
    \item[(SES)] Some sequence $\mathcal{V}$ converges.
\end{description}
Let us single out some general relations between these
properties.

\begin{lemma}\label{t3.2}
The following implications hold for an arbitrary
semicascade~$(\Omega,\varphi)$:
\begin{itemize}
    \item [(i)] \quad $(\mathrm{AEN})$ $\Rightarrow$
$\operatorname{UE} (\overline{o})$ $\Rightarrow$ $(\mathrm{AES})$.
    \item [(ii)] \quad $(\mathrm{SES})$ $\Rightarrow$ $(\mathrm{single}\, \mu_{e}$
$\mathrm{on}\,m)$.
    \item [(iii)] \quad $\operatorname{UE} (\overline{o})$
    $\Leftrightarrow$
$(\mathrm{single}$ $m \,\mathrm{in}\, \overline{o})$
$+(\operatorname{supp} \mu_{e}=m)$ $+(\mathrm{single}\, \mu_{e}$
$\mathrm{on}\,m)$.
\end{itemize}
\end{lemma}

\begin{proof}
The implication (AEN) $\Rightarrow$
$\operatorname{UE}(\overline{o})$ follows from \cite[Lemma
5.9]{16} with regard to Remark~\ref{t2.1}. The implication
$\operatorname{UE}(\overline{o})$ $\Rightarrow$ (AES) was
established in \cite[Theorem 3.2]{20}. Claim (ii) slightly
generalizes Theorem~5.4 in~\cite{19}. If there exists a convergent
ergodic operator sequence $V_{n}=U^{*}_{n}$ and the set
$m\subseteq \Omega$ is minimal, then $(U_{n}x)(\omega)\rightarrow
\overline{x}(\omega)$ for any $x\in X$ and $\omega\in m$, and
$\overline{x}(\varphi\omega)\equiv \overline{x}(\omega)$ on $m$.
Since the all orbits $o(\omega)\subseteq m$ are dense, it follows
that the restriction $\overline{x}|_{m}$ is either constant or
everywhere discontinuous. The latter is impossible for a function
of the first Baire class, and the dynamical system $(m,\varphi)$
is uniquely ergodic according to, say, the separation principle in
Theorem~\ref{t3.1}. Claim~(iii) is trivial.
\end{proof}

We see that if some minimal set supports more than one ergodic
measure, then there exist no convergent ergodic sequences (even
though there always exist convergent ergodic nets). This effect
holds for some minimal analytic diffeomorphisms of the
torus~$\mathbb{T}^{2}$ which have uncountably many ergodic
measures~\cite[Corollary~12.6.4]{12}.

In the tame case, one can say much more about the
convergence of ergodic means.
\begin{theorem}\label{t3.3}
The following assertions hold for a
tame $\mathbb{N}_{0}$-system $(\Omega,\varphi)$.
\begin{itemize}
\item[(i)] Every ergodic operator net $\{V_{\alpha}\}\subset
\mathcal{L}(X^{*})$ contains a convergent ergodic sequence
$V_{\alpha(n)}$. \item[(ii)] Every ergodic operator sequence
$\{V_{n}\}\subset \mathcal{L}(X^{*})$ contains a convergent
ergodic subsequence. In particular, the Ces\'aro means contain a
convergent subsequence.
\end{itemize}
\end{theorem}

\begin{proof}
Since $\mathcal{K}_{c}=\mathcal{K}_{c}(\Omega,\varphi)$ is
compact, we assume without loss of generality that
$V_{\alpha}\overset{\mathrm{W^*O}}{\longrightarrow}Q$, where $Q\in
\operatorname{Ker} \mathcal{K}_{c}$. For the tame semicascade
$(\Omega,\varphi)$, the topological space $\mathcal{K}_{c}$ is a
Fr\'echet--Urysohn compact set, and hence the net $\{V_{\alpha}\}$
contains a sequence
$V_{\alpha(n)}\overset{\mathrm{W^*O}}{\longrightarrow}Q$, and this
sequence is ergodic, because $Q\in \operatorname{Ker}
\mathcal{K}_{c}$. Claim~(ii) follows from the sequential
compactness of~$\mathcal{K}_{c}$ and from the preservation of
ergodicity when passing to subsequences.
\end{proof}

The ergodic sequence $\{V_{\alpha(n)}\}$ in
Theorem~\ref{t3.3}\,(i) is not a subsequence of the net
$\{V_{\alpha}\}$ in general. Now let us find out how the ergodic
and dynamic properties of tame systems are related.

\begin{theorem}\label{t3.4}
A tame $\mathbb{N}_{0}$-system $(\Omega,\varphi)$ possesses
property~$(\mathrm{SES})$, and one has the equivalences

\centerline{$(\mathrm{AEN})$ $\Leftrightarrow$ $\operatorname{UE}
(\overline{o})$ $\Leftrightarrow$ $(\mathrm{AES})$
$\Leftrightarrow$ $(\mathrm{single}$ $m$
$\mathrm{in}\,\overline{o})$.}

\end{theorem}

\begin{proof}
The existence of convergent operator ergodic sequences for a tame
semicascade was established in Theorem~\ref{t3.3}. Assume that all
such sequences converge and there exist two distinct elements
$Q_{1},Q_{2}\in \operatorname{Ker}
\mathcal{K}_{c}(\Omega,\varphi)$. By Theorem~\ref{t3.3}\,(i),
there exist ergodic sequences
$V_{n}^{(1)}\overset{\mathrm{W^*O}}{\longrightarrow} Q_{1}$ and
$V_{n}^{(2)}\overset{\mathrm{W^*O}}{\longrightarrow} Q_{2}$. Then
the mixed sequence $V_{2n-1}=V_{n}^{(1)}$, $V_{2n}=V_{n}^{(2)}$ is
ergodic but divergent. Thus, property (AES) implies the relation
$\operatorname{card} \operatorname{Ker} \mathcal{K}_{c}=1$, which
is equivalent to property~(AEN) by~\cite[Theorem 3.2]{20}, and for
tame systems we have (AEN) $\Leftrightarrow$ $\operatorname{UE}
(\overline{o})$ $\Leftrightarrow$ (AES) by Lemma~\ref{t3.2}\,(i).
Finally, Theorem~4.6 in~\cite{21} provides the implication (single
$m$ in $\overline{o}$) $\Rightarrow$ (AES), and it remains to note
that $\operatorname{UE} (\overline{o})$ $\Rightarrow$ (single $m$
in $\overline{o}$).
\end{proof}

Independently, the equivalence (AEN) $\Leftrightarrow$ (single $m$
in $\overline{o}$) for $(\Omega,\varphi)\in
\mathcal{D}_{\mathrm{tm}}$ was established in
\cite[Theorem~5.10]{16}. Theorem~\ref{t3.4}, in particular,
ensures the uniquely ergodicity of minimal tame semicascades, a
fact obtained for a wider class of tame systems as early as
in~\cite{10,14}. This assertion was strengthened
in~\cite[Lemma~5.12]{16}: the uniqueness of a minimal set
$m\subseteq \Omega$ implies the uniquely ergodicity of
$(\Omega,\varphi)\in \mathcal{D}_{\mathrm{tm}}$. In this
connection, it is useful to state the following remark.

\begin{remark}\label{t3.5}
The supports of ergodic measures of tame
$\mathbb{N}_{0}$-systems are either minimal or contain more
than one minimal set.
\end{remark}

On the other hand, it follows from~\cite[Theorem~3.1]{13} that if
an arbitrary semicascade $(\Omega,\varphi)$ has a unique minimal
set and the Ces\'aro means are weakly-star convergent, then there
exists either one ergodic measure of uncountably many such
measures. The second possibility can indeed be
realized~\cite[Sec.~4]{13}.

\begin{remark}\label{t3.6}
Even for tame systems, the convergence of one ergodic sequence
does not imply the convergence of all other ergodic sequences;
i.e., $(\mathrm{SES})$ $\nRightarrow$ $(\mathrm{AES})$. Namely, a
tame Bernoulli subshift for which the Ces\'aro means are
convergent but the property (single $m$ in $\overline{o}$) is not
satisfied was constructed in~\cite[Example~5.14]{16}. By
Theorem~\ref{t3.4}, this semicascade does not satisfy
condition~$(\mathrm{AES})$ either.
\end{remark}

\section{Asymptotic Distributions of Orbits}\label{s4}

Here we transfer some constructions in~\cite{17,19} related to the
pointwise convergence on~$\Omega$ of the Ces\'aro means~$U_{n}x$
for continuous test functions $x\in X=C(\Omega)$ to arbitrary
ergodic sequences. Instead of the individual ergodic theorem
(which fails for general averaging methods), we use a~priori
information on the pointwise convergence of some generalized
ergodic means. Our main task is to establish the possibility of
decomposition of a tame dynamical system into irreducible
(ergodic) components. By $\mathcal{P}_{in}(\Omega)$ and
$\mathcal{P}_{e}(\Omega)$ we denote the subsets of
$\varphi$-invariant and $\varphi$-ergodic measures, respectively,
in $\mathcal{P}(\Omega)$, and by $X_{1}$ we denote the set of
bounded scalar functions of the first Baire class on $\Omega$. A
set $\Theta\subseteq \Omega$ is \textit{bi-invariant} if
$\varphi^{-1}\Theta=\Theta$. Further, let $D(\Omega)$ be the set
of Dirac measures $\delta_{\omega}$ on $\Omega$, and let
$\mathcal{K}_{c}=\mathcal{K}_{c}(\Omega,\varphi)\subseteq
\mathcal{L}(X^{*})$ be the operator semigroup defined in
Sec.~\ref{s2.2}.

We assume the existence of a convergent ergodic operator
sequence $\mathcal{V}=\{V_{n}\}\subset \mathcal{L}(X^{*})$,
$V_{n} \overset{\mathrm{W^*O}}{\longrightarrow}Q\in
\operatorname{Ker} \mathcal{K}_{c}$, and write this
convergence briefly as $\mathcal{V}\rightarrow Q$. In this
case, $(U_{n}x)(\omega)\rightarrow \overline{x}(\omega)$
for the dual ergodic sequence $\{U_{n}\}\subset
\mathcal{L}(X)$, $U_{n}^{*}=V_{n},$ and all $\omega\in
\Omega$ and $x\in X$; furthermore, the function
$\overline{x} \in X_{1}$ is invariant
$(\overline{x}\circ\varphi=\overline{x})$, and to each
point $\omega\in \Omega$ there corresponds a measure
$\mu_{\omega} = Q\delta_{\omega}\in
\mathcal{P}_{in}(\Omega)$,
$\overline{x}(\omega)=(x,\mu_{\omega})$, determining the
asymptotic $\mathcal{V}$-distribution of the orbit
$o(\omega)$. Essentially, this means that
$V_{n}\delta_{\omega}\stackrel{w^{*}}{\longrightarrow}\mu_{\omega}$.
A linear projection~$Q$ in $X^{*}$ induces a mapping
$\Psi_{\mathcal{V}}:\Omega\rightarrow
\mathcal{P}_{in}(\Omega)$ of the first Baire class. The
notation $\Psi_{\mathcal{V}}$ is convenient, even though
this mapping is completely determined by the limit
element~$Q$ of the sequence~$\mathcal{V}$.

\begin{lemma}\label{t4.1}
If an ergodic sequence
$\mathcal{V}=\{V_{n}\}$ is convergent,
then $\Psi_{\mathcal{V}}\Omega
\supseteq\mathcal{P}_{e}(\Omega)$.
\end{lemma}

In other words, for any convergent ergodic sequence
$\mathcal{V}$, every ergodic measure determines the
asymptotic $\mathcal{V}$-distribution of some orbit.

\begin{proof}
Let $\mu\in \mathcal{P}_{e}(\Omega)$, $x\in X$, $U^{*}_{n}=V_{n}$
and $c(x)=(x,\mu)$. Under the assumptions of the lemma,
$(U_{n}x,\mu)=(x,V_{n}\mu)=(x,\mu)$, the dual sequence
$(U_{n}x)(\omega)$ converges to $\overline{x}(\omega)$ for all
$\omega\in \Omega$, $\overline{x}=\overline{x}\circ\varphi$, and
$(\overline{x},\mu)=c(x)$ by the Lebesgue theorem. An argument
like \cite[Proposition 7.15, (i)~$\Rightarrow$~(iv)]{3} shows
that, by virtue of the ergodicity of~$\mu$, the bounded invariant
function $\overline{x}\in X_{1}$ is identically equal to a
constant $c(x)$ on some Borel set $\Theta_{x,\mu}\subseteq \Omega$
of full $\mu$-measure. Consequently, for the points $\omega \in
\Theta_{x,\mu}$ we have
\begin{equation}\label{e4.1}
(x,V_{n}\delta_{\omega})=(U_{n}x,\delta_{\omega})\rightarrow
(\overline{x},\delta_{\omega})=(\overline{x},\mu)=(x,\mu).
\end{equation}
Now we take $x$ from an arbitrary countable set~$Y$
everywhere dense in~$X$ and obtain relations~\eqref{e4.1}
for $x\in Y$ and $\omega\in \Theta_{\mu}$, where
$\Theta_{\mu}=\bigcap_{x\in Y}\Theta_{x,\mu}$ and
$\mu(\Theta_{\mu})=1$. Since $\|V_{n}\|\leq 1$ for all
$n\in \mathbb{N}_{0}$, it follows that the same is true for
any $x\in X$ and $\omega\in \Theta_{\mu}$. Thus,
$V_{n}\delta_{\omega}\stackrel{\operatorname{w^{*}}}{\longrightarrow}\mu$
for $\omega\in \Theta_{\mu}$.
\end{proof}

For a convergent ergodic sequence $\mathcal{V}=\{V_{n}\}$,
we set
$$
\Omega_{\mathcal{V}}=\{\omega\in \Omega: \mu_{\omega}\in
\mathcal{P}_{e}(\Omega)\},
$$
where
$V_{n}\delta_{\omega}\stackrel{\operatorname{w^{*}}}{\longrightarrow}
\mu_{\omega}$; then, for \textit{ergodic} measures~$\mu$,
the subsets
$\Omega_{\mu,\mathcal{V}}=\Psi^{-1}_{\mathcal{V}}\mu$ form
partitions of $\Omega_{\mathcal{V}}$ into
$\mathcal{V}$-quasi-ergodic components. The sets
$\Omega_{\mathcal{V}}$ and $\Omega_{\mu,\mathcal{V}}$ are
bi-invariant. They are Borel sets, which follows by a
purely topological reasoning \cite[pp.~119--120]{19} not
related in any way to the specific features of Ces\'aro
averaging. Since $\Omega_{\mu,\mathcal{V}}\supseteq
\Theta_{\mu}$, where $\Theta_{\mu}$ is the set in the proof
of Lemma~\ref{t4.1}, we obtain the following assertion.

\begin{corollary}\label{t4.2}
If an ergodic sequence $\mathcal{V}$
is convergent, then to each ergodic measure~$\mu$ there
corresponds a Borel $\mathcal{V}$-quasi-ergodic set
$\Omega_{\mu,\mathcal{V}}$ of full $\mu$-measure.
\end{corollary}

Now let us discuss the main topic of the present paper.

\begin{definition}\label{t4.3}
We say that an $\mathbb{N}_{0}$-system $(\Omega,\varphi)$
is \textit{ergodically decomposable} if there exists a
convergent operator ergodic sequence~$\mathcal{V}$ such
that $\Omega_{\mathcal{V}}=\Omega$, or, equivalently,
$\Psi_{\mathcal{V}}\Omega =\mathcal{P}_{e}(\Omega)$.
\end{definition}

In this case, the \textit{topological dynamical system}
$(\Omega,\varphi)$ essentially admits a decomposition into ergodic
subsystems $(\Omega_{\mu,\mathcal{V}},\varphi)$, $\mu\in
\mathcal{P}_{e}(\Omega)$. Here $\mathcal{V}\rightarrow Q\in
\operatorname{Ker} \mathcal{K}_{c}$, and to each continuous
function $x\in X$ there corresponds a function
$\overline{x}=Q^{*}x\in X_{1}$ taking a constant value
$(\overline{x},\mu)=(x,\mu)$ on each quasi-ergodic set
$\Omega_{\mu,\mathcal{V}}$. Thus, for each measure $\mu\in
\mathcal{P}_{e}(\Omega)$, the \textit{measure-preserving dynamical
system} $(\Omega_{\mu,\mathcal{V}},\varphi)$ is ergodic with
respect to $\mu$ in the standard sense~\cite[Definition~6.18]{3}.
In the interpretation given in~\cite[Sec.~4.1]{12}, the ergodic
decomposability of the semicascade $(\Omega,\varphi)$ means that
the asymptotic $\mathcal{V}$-distributions of all orbits are
determined by ergodic measures. Note also that the mapping
$\Psi_{\mathcal{V}}:\Omega\rightarrow \mathcal{P}_{e}(\Omega)$
inducing the decomposition of $(\Omega,\varphi)$ is a sequential
pointwise limit of continuous mappings and hence belongs to the
first Baire class, so that its points of continuity form a dense
$G_{\delta}$-set in~$\Omega$.

It turns out that the ergodic decomposability on an
$\mathbb{N}_{0}$-dynamical system is related to the
existence of operator ergodic sequences converging to the
extreme points of the kernel of the semigroup
$\mathcal{K}_{c}=\mathcal{K}_{c}(\Omega,\varphi)$.

\begin{proposition}\label{t4.4}
If an ergodic sequence $\mathcal{V}$ converges to
$Q\in\operatorname{ex} \operatorname{Ker} \mathcal{K}_{c},$
then the dynamical system $(\Omega,\varphi)$ is ergodically
decomposable.
\end{proposition}

\begin{proof}
Let $\mathcal{V}=\{V_{n}\}$. By
\cite[Proposition~2.10]{21}, one has
$Q:D(\Omega)\rightarrow \mathcal{P}_{e}(\Omega)$, and since
$V_{n}\overset{\mathrm{W^*O}}{\longrightarrow}Q$, it
follows that
$V_{n}\delta_{\omega}\stackrel{\operatorname{w^{*}}}{\longrightarrow}\mu\in
\mathcal{P}_{e}(\Omega)$ for each point $\omega\in\Omega$.
Thus, $\Omega_{\mathcal{V}}=\Omega$, and the system
$(\Omega,\varphi)$ is ergodically decomposable.
\end{proof}

\begin{theorem}[main theorem]\label{t4.5}
Tame
$\mathbb{N}_{0}$-systems $(\Omega,\varphi)$ are ergodically
decomposable.
\end{theorem}

\begin{proof}
An arbitrary projection $Q\in \operatorname{ex} \operatorname{Ker}
\mathcal{K}_{c}$ is the W$^*$O-limit of some ergodic net
$\{V_{\alpha}\}\subset \mathcal{L}(X^{*})$. By
Theorem~\ref{t3.3}\,(i), there exists an ergodic operator sequence
$V_{\alpha(n)}\overset{\mathrm{W^*O}}{\longrightarrow}Q$. The
desired assertion with $\mathcal{V}=\{V_{\alpha(n)}\}$ follows
from Proposition~\ref{t4.4}.
\end{proof}

Now let us describe the structure of all possible
decompositions of tame systems into ergodic components.

\begin{lemma}\label{t4.6}
For a tame $\mathbb{N}_{0}$-system
$(\Omega,\varphi)$, the operators $T\in
\mathcal{K}_{c}(\Omega,\varphi)$ are determined by their
values on Dirac measures.
\end{lemma}

\begin{proof}
Under the assumptions of the lemma, the semigroup
$\mathcal{K}_{c}$ is a Fr\'echet--Urysohn compact set, and hence
for each $T\in \mathcal{K}_{c}$ there exists a sequence
$\{V_{n}\}\subseteq \operatorname{co} \{V^{n}$, $n\geq 0\}$
converging to~$T$ in the W$^*$O-topology of the space
$\mathcal{L}(X^{*})$. If $U_{n}^{*}=V_{n}$, $x\in X$ and
$\omega\in \Omega$, then
$(x,V_{n}\delta_{\omega})=(U_{n}x)(\omega)$ and
$(U_{n}x)(\omega)\rightarrow \overline{x}(\omega)$. Here
$\overline{x}\in X_{1}$, and $(x,V_{n}\mu)=(U_{n}x,\mu)\rightarrow
(\overline{x},\mu)=(x,T\mu)$ for each measure $\mu\in
\mathcal{P}(\Omega)$ by the Lebesgue theorem. At the same time $
\overline{x}(\omega)=(x,T\delta_{\omega})$ for $\omega\in \Omega$,
hence the operator $T$ is completely defined by its values on
$D(\Omega)$.
\end{proof}

Lemma 4.6 is a refinement of a similar assertion
\cite[Theorem~3.5, (a1)$\,\Rightarrow$ (a4)]{21} where instead of
the condition $(\Omega,\varphi)\in \mathcal{D}_{\mathrm{tm}}$ the
enveloping semigroup $E(\Omega,\varphi)$ is required to be
metrizable.

\begin{corollary}\label{t4.7}
If $(\Omega,\varphi)\in \mathcal{D}_{\mathrm{tm}}$ and
$Q_{1}|_{D(\Omega)}=Q_{2}|_{D(\Omega)}$ for $Q_{1},Q_{2}\in
\operatorname{Ker} \mathcal{K}_{c}$, then $Q_{1}=Q_{2}$.
\end{corollary}

Hence we readily find that in the tame case the condition
$Q\in \operatorname{ex} \operatorname{Ker} \mathcal{K}_{c}$
is not only sufficient but also necessary for the relation
$Q:D(\Omega)\rightarrow \mathcal{P}_{e}(\Omega)$, $Q\in
\operatorname{Ker} \mathcal{K}_{c}$ to hold. For
$(\Omega,\varphi)\in \mathcal{D}_{\mathrm{tm}}$, it is
natural to define quasi-ergodic sets based on elements
$Q\in \operatorname{ex} \operatorname{Ker} \mathcal{K}_{c}$ rather
than convergent ergodic sequences~$\mathcal{V}$. Namely, we
set
$$
\Omega_{\mu,Q}=\{\omega\in \Omega:Q\delta_{\omega}=\mu \}, \quad
\mu\in \mathcal{P}_{e}(\Omega).
$$
We see that Borel bi-invariant quasi-ergodic sets $\Omega_{\mu,Q}$
of full $\mu$-measure form a partition~$\Phi_{Q}$ of the phase
space~$\Omega$. The set~$\Lambda$ of all ergodic sequences
$\mathcal{V}\rightarrow Q$ splits into disjoint
classes~$\Lambda_{Q}$ corresponding to distinct~$Q$. The elements
$\mathcal{V}\in \Lambda_{Q}$ define a relationship between the
dynamics of the semicascade $(\Omega,\varphi)$ and the ergodic
measures; namely, the asymptotic $\mathcal{V}$-distribution of
each orbit~$o(\omega)$ is determined by the measure
$\mu=Q\delta_{\omega}$. There exists a one-to-one (by virtue of
Corollary~\ref{t4.7}) correspondence between the projections $Q\in
\operatorname{ex} \operatorname{Ker} \mathcal{K}_{c}$, the
partitions~$\Phi_{Q}$ of the phase space~$\Omega$ into
quasi-ergodic sets, and the partitions~$\Lambda_{Q}$ of~$\Lambda$
converging to the extreme points of the kernel of the semigroup
$\mathcal{K}_{c}(\Omega,\varphi)$ of ergodic operator sequences.

\section{Supplement}\label{s5}

Here we consider a number of typical examples of tame and untame
$\mathbb{N}_{0}$-dynamical systems. Let $I=[0,1]$.
\begin{itemize}
\item[(1)] According to~\cite[Proposition 10.5]{8}
and~\cite[Sec.~9]{6}, every semicascade generated by a
homeomorphism of~$I$ or~$\mathbb{S}^{1}$ has a metrizable Ellis
semigroup and hence is tame.
\item[(2)] The left Bernoulli shift on the set
$\Omega=\{0,1\}^{\mathbb{N}_{0}}$ of sequences
$\omega_{0},\omega_{1},\dots$ with the standard metric
$\rho(\omega,\nu)=(1+\min\{k:\omega_{k}\neq \nu_{k}\})^{-1}$
generates a untame $\mathbb{N}_{0}$-system $(\Omega,\varphi)$,
which, however, admits tame subsystems $(\Theta,\varphi)$. Here is
an elegant description of these systems: every infinite set $L
\subseteq \mathbb{N}_{0}$ contains an infinite subset $K\subseteq
L$ such that the projection $\pi_{K}(\Theta)$ is a countable
subset of $\{0,1\}^{K}$~\cite[Theorem~4.7]{9}.
\item[(3)] The set of periodic points of the semicascade
$(I,\varphi)$ in the example in \cite[pp.~147--149]{1} is
nonclosed, and every orbit $o(\omega)$,$\omega\in I,$ is either
eventually periodic ($\varphi^{k}\omega=\varphi^{k+p}\omega$ for
some $k\geq 0$, $p\geq 1$) or its limit points fill the classical
Cantor set. This semicascade proves to be tame~\cite[Example
5.8\,(c)]{15}.
\item[(4)] On the other hand, every semicascade $(I,\varphi)$
admitting periodic points with period that is not a power of~$2$
is not tame \cite[Example 5.8\,(e)]{15}.
\item[(5)] A slight modification of the argument
in~\cite[p.~2354]{6} shows that the projective action of an
arbitrary invertible operator $T\in GL(n,\mathbb{R}^{n})$, $n\geq
2$, induces a tame semicascade on the sphere~$\mathbb{S}^{n-1}$.
\end{itemize}

Examples~(3) and~(5) show that tame systems can have a rather
nontrivial phase dynamics.

Recently, using function--theoretic argument, Lebedev obtained a
criteria that allows to distinguish tame and untame affine
endomorphisms of the torus $\varphi:\omega\rightarrow
A\omega+b\quad(\omega\in \mathbb{T}^{d},\;d\geq 1)$ with an
integer matrix $A$ and an arbitrary shift $b\in \mathbb{T}^{d}$.
If $\operatorname{det} A=\pm 1$, then $\varphi$ is an
automorphism.

\begin{theorem}[Lebedev~\cite{18}]\label{t5.1}
A semicascade $(\mathbb{T}^{d},\varphi)$ is tame if and only if
$A^{k}=A^{l}$ for some $k,l\in \mathbb{N}_{0}$, $k\neq l$.
\end{theorem}

If $\operatorname{det}A=\pm 1$, then the conclusion of the theorem
is $A^{k}=\operatorname{Id}$. In particular, the automorphism
$\varphi:(\omega_{1},\omega_{2})\rightarrow
(\omega_{1}+\omega_{2},\omega_{2})$ of the torus $\mathbb{T}^{2}$
is not tame.

{\bf Acknowledgements.} The author is thankful to V. Lebedev, H.
Kreidler, and M. Megrelishvili for ideas, helpful suggestions, and
inspiring discussions.

\end{document}